\documentclass[final,3p,times]{elsarticle}

\usepackage{amsthm,amsmath,amssymb,amsfonts,dsfont}
\usepackage{graphicx,graphics,color}
\usepackage{subeqnarray}
\usepackage{cases}
\newtheorem{theorem}{Theorem}[section]

\journal{}

\begin{document}

\begin{frontmatter}



\title{Global Attractivity of a Nonlocal Delayed Diffusive Dengue Model in a Spatially Homogeneous Environment}
\author[HLJU]{Xue Ren}
\ead{xueren@hlju.edu.cn}

\author[HLJU]{Ran Zhang\corref{mycorrespondingauthor}}
\cortext[mycorrespondingauthor]{Corresponding author}
\ead{ranzhang@hlju.edu.cn}

\address[HLJU]{\scriptsize Engineering Research Center of Agricultural Microbiology Technology, Ministry of Education; Heilongjiang Provincial Key Laboratory of Ecological Restoration and Resource Utilization for Cold Region; School of Mathematical Science, Heilongjiang University,\ Harbin 150080, P.R. China}

\begin{abstract}
In Xu and Zhao (2015), the global attractivity of positive constant steady state is established through the application of the fluctuation method, subject to the sufficient condition that the disease will stabilized at the unique spatially-homogeneous steady state if $\Re_0>1$ exceeds a certain threshold. The focus of this study is to eliminate the need for a sufficient condition by employing a suitable Lyapunov functional and prove that the positive constant steady state is globally attractive when $\Re_0$ is exactly greater than unity, which significantly enhancing the findings outlined in Theorem 3.3 of Xu and Zhao (2015).
\end{abstract}


\end{frontmatter}


\section{Introduction}\label{sec:Inc}
\def\d {{\rm d}}
Dengue is a viral disease transmitted by mosquitoes, primarily the female \textit{Aedes aegypti} or \textit{Aedes albopictus}, which has quickly proliferated among humans through mosquito bites.
According to the WHO, the global incidence of dengue has markedly increased in recent decades. Approximately half of the world's population is now vulnerable to this disease, with an estimated 100-400 million infections occurring annually.
Dengue fever is prevalent throughout the tropics, the risk of dengue fever varies in some areas, which are affected by rainfall, temperature, relative humidity and unplanned rapid urbanization \cite{WTO2020}.
Recently,
Xu and Zhao \cite{XuZhaoAMC2015} studied a dengue model with nonlocal delayed reaction in both heterogeneous and homogeneous environments.
In the heterogeneous scenario, the authors present findings on disease extinction and persistence linked to its basic reproduction number.
For the spatially homogeneous case, a series of sufficient conditions is provided to demonstrate the global attractiveness of the positive steady state within a bounded domain. Additionally, the study in \cite{XuZhaoAMC2015} explores the traveling wave problem in an unbounded domain. The spatially homogeneous model proposed in \cite{XuZhaoAMC2015} reads as:
\begin{equation}
\label{Model}\left\{
\begin{array}{ll}
\vspace{2mm}
\displaystyle   \frac{\partial u_1}{\partial t} = d_m \Delta u_1 + \beta_m(A-u_1)\int_\Omega \Gamma(d_m\tau_a,x,y)u_3(t-\tau_a,y)\d y - \mu_m u_1,\\
\vspace{2mm}
\displaystyle   \frac{\partial u_2}{\partial t} = d_h \Delta u_2 + H - \beta_h u_1 u_2 - \mu_h u_2,\\
\displaystyle   \frac{\partial u_3}{\partial t} = d_h \Delta u_3 + \beta_h \mathrm{e}^{-\mu_h \tau_b} \int_\Omega \Gamma(d_h\tau_b,x,y)u_1(t-\tau_b,y)u_2(t-\tau_b,y)\d y - \rho_h u_3,
\end{array}\right.
\end{equation}
with initial condition and Neumann boundary conditions:
\[\left\{
\begin{array}{ll}
\vspace{2mm}
\displaystyle   u_i(s,x) = \phi_i(0,x) \geq 0,\ \ s\in[-\max\{\tau_a,\tau_b\},0],\ \ x\in\Omega,\\
\vspace{0mm}
\displaystyle   \frac{\partial u_i}{\partial \textbf{n}} = 0,\ \ t >0,\ \ x \in \partial \Omega,
\end{array}\right.
\]
for $i=1,2,3$, where $u_i: = u_i(t,x)$ denotes the densities of infectious mosquitos, susceptible and infectious humans, respectively. $\beta_m: = bp$, $\beta_h: = bq$ and $\rho_h: = \mu_h+\gamma_h$. Constants $d_m$ and $d_h$ are the diffusion rate of mosquitos and human, respectively. $A$ and $H$ are the rates at which new individuals (mosquitoes or humans) are added to the population.
$b$ is the average rate at which mosquitoes bite hosts,
$p\ (q)$ is the successful infection rate though a bite from a susceptible (or infectious) mosquito to an infectious (or susceptible).
$\gamma_h$ represents the rate at which infected individuals recover from the infection. The removal rate of each population are assumed to be $\mu_m$ and $\mu_h$. Delays $\tau_a$ and $\tau_b$ are the incubation periods.
{$\Gamma(t,x,y)$ denotes the Green function associated with $\frac{\partial u}{\partial t} = \Delta u$ subject to the Neumann boundary condition.}
All coefficients are assumed to be positive. We first give some notions of spaces, which are the same in \cite{XuZhaoAMC2015}. Let $\mathds{X}:=\textrm{BUC}(\overline{\Omega},\mathds{R}^3)$ be the collection of all functions that are both bounded and uniformly continuous from the closure of ${\Omega}$ to $\mathds{R}^3$ and $\mathds{C} = C([ -\max\{\tau_a,\tau_b\},0],\mathds{X})$.
Denote
\[
C_\mathbf{M} = \left\{\phi \in \mathds{C}:\ \textbf{0}\leq\phi(t,x)\leq \textbf{M},\ \forall x\in \overline{\Omega},\ \forall t \in[ -\max\{\tau_a,\tau_b\},0]\right\}
\]
with
\[
\textbf{M}:= \left(A,\ \frac{H}{\mu_h},\ \frac{AH\beta_h}{\mu_h \rho_h}\mathrm{e}^{-\mu_h \tau_b}\right).
\]
According to Theorem 3.1 of \cite{XuZhaoAMC2015}, system (\ref{Model}) admits a unique mild solution $u(t,\cdot,\phi)\in C_\textbf{M}$ for all $t \geq 0$ with $u(0,\cdot,\phi) = \phi$. Moreover, the solution semi-flow defined by $\Phi(t):=u_t(\cdot):C_\textbf{M}\rightarrow C_\textbf{M}$ has a global compact attractor.

It has been proved in \cite{XuZhaoAMC2015} that the positive steady state of system (\ref{Model}) is globally attractive, which is achieved by using the fluctuation method \cite{ThiemeZhaoNARWA2001}. However, the global attractivity of the positive steady state requires the basic reproduction number $\Re_0$ of (\ref{Model}) satisfying the additional condition. The corresponding result in \cite{XuZhaoAMC2015} reads as:

\begin{theorem}\label{orginal-thm}{\cite[Theorem 3(iii)]{XuZhaoAMC2015}}
Let $\Re_0 = \sqrt{\frac{\beta_h\beta_mAH\mathrm{e}^{-\mu_h\tau_b}}{\mu_h\mu_m\rho_h}}$. If 
\begin{equation}\label{condition}
\Re_0>\max\left\{1,\sqrt{\frac{A\beta_h}{\mu_h}}\right\},
\end{equation} 
then the system (\ref{Model}) admits a unique constant steady
state $u^* = (u_1^*,u_2^*,u_3^*)^\mathrm{T}$ such that 
\[
\lim_{t\rightarrow\infty} u(t,x,\phi) = u^*
\]
uniformly for $x\in \overline{\Omega}$, provided that $\phi\in C_\mathbf{M}$ with $\phi_1(0,\cdot)\not\equiv 0$ or $\phi_3(0,\cdot)\not\equiv 0$.
\end{theorem}
It is well known that $\Re_0$ is a critical threshold in epidemic modeling, dictates whether steady states are globally attractive. Theorem \ref{orginal-thm} prompts the question addressed in our paper: is the endemic steady state of model (\ref{Model}) globally attractive when $\Re_0$ is precisely greater than unity? This inquiry motivates our work in Section 2, where we will resolve this issue.

\section{Main results}
The objective of this paper is to investigate the global attractiveness of the positive steady state of system \eqref{Model} while relaxing the restrictive condition \eqref{condition}, which will notably enhance Theorem \ref{orginal-thm}. It is well-established that employing suitable Lyapunov functions (or functionals) is an effective approach to achieving global attractivity of endemic equilibrium in epidemic models governed by ordinary differential equations and other related frameworks, including functional differential equations and fractional order differential equations \cite{KorobeinikovWakeAML2002, HuangBMB2010, YangXuAML2020}.
Motivated by recent work \cite{LiZhaoJDE2021,ZhangWangJMB2022}, we will establish the following result.

\begin{theorem}
If $\Re_0>1$, then the system (\ref{Model}) admits a unique constant steady state $u^*$ such that
\[
\lim_{t\rightarrow\infty} u(t,x,\phi) = u^*
\]
uniformly for $x\in \overline{\Omega}$, provided that $\phi\in C_\mathbf{M}$ with $\phi_1(0,\cdot)\not\equiv 0$ or $\phi_3(0,\cdot)\not\equiv 0$..
\end{theorem}

\begin{proof}
Let $g(\omega)=\omega-1-\ln \omega$ for $\omega>0$, it is easy to check $g(1)=0$ is the global minimum value of $g(\omega)$ for $\omega>0$.
Define the following Lyapunov functional:
\[
V(t) = \int_\Omega \left(\sum_{i=1}^3 L_i(t,x) + \sum_{i=1}^2 W_i(t,x)\right) \d x
\]
with
\[
L_1(t,x) = \frac{\beta_h u_1^* u_2^*}{\mu_m} g \left(\frac{u_1}{u_1^*}\right),\ \
L_2{(t,x)} = u_2^* g \left(\frac{u_2}{u_2^*}\right),\ \ L_3{(t,x)} = \mathrm{e}^{\mu_h \tau_b} u_3^* g \left(\frac{u_3}{u_3^*}\right),
\]
\[
W_1(t,x) = \beta_h u_1^* u_2^* \int_{-\tau_a}^0 \int_\Omega \Gamma(d_m(-\theta),x,y) g \left(\frac{u_{3,\theta}({y})}{u_3^*}\right)\d y\d \theta
\]
and
\[
W_2(t,x) = \beta_h u_1^* u_2^* \int_{-\tau_b}^0 \int_\Omega \Gamma(d_h(-\theta),x,y) g \left(\frac{(u_1\diamond u_2)(t+\theta,{y})}{u_1^*u_2^*}\right)\d y\d \theta.
\]
where $u_{i,\theta}(y): = u_i(t+\theta,y)$, $u \diamond v (\xi,\zeta): = u(\xi,\zeta)v(\xi,\zeta)$.
The differential of $L_1(t,x)$ is calculating as follows
\begin{align*}
\frac{\partial L_1(t,x)}{\partial t} = &\ \frac{\beta_h u_2^*}{\mu_m} \left(1-\frac{u_1^*}{u_1}\right)\left(d_m \Delta u_1 + \beta_m(A-u_1)\int_\Omega \Gamma(d_m\tau_a,x,y)u_{3,-\tau_a}(y)\d y - \mu_m u_1\right)\\
=&\ \frac{\beta_h u_2^*}{\mu_m} \left(1-\frac{u_1^*}{u_1}\right)\left(d_m \Delta u_1 - \mu_m u_1\right)\\
&\ + \frac{\beta_h u_2^*}{\mu_m} \left(1-\frac{u_1^*}{u_1}\right)\beta_m(A-u_1^*)\int_\Omega \Gamma(d_m\tau_a,x,y)u_{3,-\tau_a}(y)\d y\\
&\ + \frac{\beta_h u_2^*}{\mu_m} \left(1-\frac{u_1^*}{u_1}\right)\beta_m(u_1^*-u_1)\int_\Omega \Gamma(d_m\tau_a,x,y)u_{3,-\tau_a}(y)\d y.
\end{align*}
Note that $\beta_m(A-u_1^*)u_3^* = \mu_m u_1^*$, we yield
\begin{align*}
\frac{\partial L_1(t,x)}{\partial t} = &\ \frac{d_m\beta_h u_2^*}{\mu_m} \left(1-\frac{u_1^*}{u_1}\right) \Delta u_1 - \beta_hu_1^*u_2^*\frac{u_1}{u_1^*} + \beta_hu_1^*u_2^* - \frac{\beta_m \beta_h u_2^*}{\mu_m u_1} (u_1-u_1^*)^2{\int_\Omega \Gamma(d_m\tau_a,x,y)u_{3,-\tau_a}(y)\d y}\\
&\ + \beta_hu_1^*u_2^* \int_\Omega \Gamma(d_m\tau_a,x,y)\frac{u_{3,-\tau_a}(y)}{u_3^*}\d y {-} \beta_hu_1^*u_2^* \int_\Omega \Gamma(d_m\tau_a,x,y)\frac{u_{3,-\tau_a}(y)u_1^*}{u_3^*u_1}\d y.
\end{align*}
Using the relationship of epidemic equilibrium $H=\beta_h u_1^* u_2^* + \mu_h u_2^*$ and $\beta_h e^{-\mu_h \tau_b} u_1^*u_2^* = \rho_h u_3^*$, we obtain the derivatives of $L_2(t,x)$ and $L_3(t,x)$ as follows
\begin{align*}
\frac{\partial L_2(t,x)}{\partial t} = &\ \left(1-\frac{u_2^*}{u_2}\right) \left(d_h \Delta u_2 + H - \beta_h u_1u_2 - \mu_h u_2\right)\\
= &\ \left(1-\frac{u_2^*}{u_2}\right) \left(d_h \Delta u_2 + \beta_h u_1^*u_2^* + \mu_h u_2^* - \beta_h u_1u_2 - \mu_h u_2\right)\\
= &\ \left(1-\frac{u_2^*}{u_2}\right) d_h \Delta u_2 - \frac{\mu_h}{u_2}(u_2-u^*)^2 + \beta_hu_1^*u_2^* - \beta_hu_1u_2 - \beta_hu_1^*u_2^* \frac{u_2^*}{u_2} + \beta_hu_1u_2^*,
\end{align*}
and
\begin{align*}
\frac{\partial L_3(t,x)}{\partial t} = &\ \mathrm{e}^{\mu_h \tau_b}\left(1-\frac{u_3^*}{u_3}\right) \left(d_h \Delta u_3 + \beta_h \mathrm{e}^{-\mu_h \tau_b} \int_\Omega \Gamma(d_h\tau_b,x,y)(u_1\diamond u_2)(t-\tau_b,y)\d y - \rho_h u_3\right)\\
= &\ \mathrm{e}^{\mu_h \tau_b}\left(1-\frac{u_3^*}{u_3}\right) d_h \Delta u_3 + \beta_hu_1^*u_2^* - \beta_hu_1^*u_2^* \frac{u_3}{u_3^*}\\
&\ + \beta_hu_1^*u_2^* \int_\Omega \Gamma(d_h\tau_b,x,y)\left(1-\frac{u_3^*}{u_3}\right)\frac{(u_1\diamond u_2)(t-\tau_b,y)}{u_1^*u_2^*}\d y.
\end{align*}
Since system (\ref{Model}) admits Neumann boundary condition, we have $\int_\Omega \Delta u_i\d x = 0$ and $\int_\Omega \frac{\Delta u_i}{u_i}\d x = \int_\Omega \frac{|\nabla u_i|^2}{u_i^2}\d x$. Consequently,
\begin{align*}
\frac{\partial }{\partial t} \int_\Omega \sum_{i=1}^3 L_i(t,x) \d x = &\ -\frac{d_m\beta_h {u_1^*}u_2^*}{\mu_m} \int_\Omega \frac{|\nabla u_1|^2}{u_1^2}\d x
- d_h {u_2^*} \int_\Omega \frac{|\nabla u_2|^2}{u_2^2}\d x - \mathrm{e}^{\mu_h \tau_b}d_h {u_3^*} \int_\Omega \frac{|\nabla u_3|^2}{u_3^2}\d x
+ \int_\Omega \Xi(t,x) \d x.
\end{align*}
where
\begin{align*}
\Xi(t,x) = &\  - \frac{\beta_m \beta_h u_2^*}{\mu_m u_1} (u_1-u_1^*)^2{\int_\Omega \Gamma(d_m\tau_a,x,y)u_{3,-\tau_a}(y)\d y} - \frac{\mu_h}{u_2}(u_2-u_2^*)^2 - \beta_hu_1^*u_2^*\frac{u_1}{u_1^*} + \beta_hu_1^*u_2^* \\
&\ + \beta_hu_1^*u_2^* \int_\Omega \Gamma(d_m\tau_a,x,y)\frac{u_{3,-\tau_a}(y)}{u_3^*}\d y + \beta_hu_1^*u_2^* \int_\Omega \Gamma(d_m\tau_a,x,y)\frac{u_{3,-\tau_a}(y)u_1^*}{u_3^*u_1}\d y\\
&\ + \beta_hu_1^*u_2^* - \beta_hu_1u_2 - \beta_hu_1^*u_2^* \frac{u_2^*}{u_2} + \beta_hu_1u_2^* + \beta_hu_1^*u_2^* - \beta_hu_1^*u_2^* \frac{u_3}{u_3^*}\\
&\ + \beta_hu_1^*u_2^* \int_\Omega \Gamma(d_h\tau_b,x,y)\left(1-\frac{u_3^*}{u_3}\right)\frac{(u_1\diamond u_2)(t-\tau_b,y)}{u_1^*u_2^*}\d y.
\end{align*}
Next, we deal with $W_1$ and $W_2$, one can obtain that
\begin{align*}
\frac{\partial W_1(t,x)}{\partial t} = &\ \beta_h u_1^* u_2^* \frac{\partial}{\partial t} \int_{-\tau_a}^0 \int_\Omega \Gamma(d_m(-\theta),x,y) g \left(\frac{u_{3,\theta}(y)}{u_3^*}\right)\d y\d \theta\\
=&\ \beta_h u_1^* u_2^* \frac{\partial}{\partial \theta} \int_{-\tau_a}^0 \int_\Omega \Gamma(d_m(-\theta),x,y) g \left(\frac{u_{3,\theta}(y)}{u_3^*}\right)\d y\d \theta\\
=&\ \beta_h u_1^* u_2^* g \left(\frac{u_3}{u_3^*}\right) - \beta_h u_1^* u_2^* \int_\Omega \Gamma(d_m\tau_a,x,y) g \left(\frac{u_{3,-\tau_a}(y)}{u_3^*}\right)\d y.
\end{align*}
Similarly,
\begin{align*}
\frac{\partial W_2(t,x)}{\partial t} = \beta_h u_1^* u_2^* g \left(\frac{u_1u_2}{u_1^*u_2^*}\right) - \beta_h u_1^* u_2^* \int_\Omega \Gamma(d_h\tau_b,x,y) g \left(\frac{(u_1\diamond u_2)(t-\tau_b,y)}{u_1^*u_2^*}\right)\d y.
\end{align*}
Thus,
\begin{align*}
\frac{\d V(t)}{\d t} = &\ -\frac{d_m\beta_h u_2^*}{\mu_m} \int_\Omega \frac{|\nabla u_1|^2}{u_1^2}\d x
- d_h \int_\Omega \frac{|\nabla u_2|^2}{u_2^2}\d x - \mathrm{e}^{\mu_h \tau_b}d_h \int_\Omega \frac{|\nabla u_3|^2}{u_3^2}\d x
+ \int_\Omega \Xi(x) \d x\\
&\ + \int_\Omega \beta_h u_1^* u_2^* g \left(\frac{u_3}{u_3^*}\right) \d x - \int_\Omega\beta_h u_1^* u_2^* \int_\Omega \Gamma(d_m\tau_a,x,y) g \left(\frac{u_{3,-\tau_a}(y)}{u_3^*}\right)\d y \d x\\
&\ + \int_\Omega\beta_h u_1^* u_2^* g \left(\frac{u_1u_2}{u_1^*u_2^*}\right)\d x - \int_\Omega\beta_h u_1^* u_2^* \int_\Omega \Gamma(d_h\tau_b,x,y) g \left(\frac{(u_1\diamond u_2)(t-\tau_b,y)}{u_1^*u_2^*}\right)\d y\d x\\
=&\ -\frac{d_m\beta_h u_2^*}{\mu_m} \int_\Omega \frac{|\nabla u_1|^2}{u_1^2}\d x
- d_h \int_\Omega \frac{|\nabla u_2|^2}{u_2^2}\d x - \mathrm{e}^{\mu_h \tau_b}d_h \int_\Omega \frac{|\nabla u_3|^2}{u_3^2}\d x
- \int_\Omega\beta_h u_1^* u_2^* g\left(\frac{u_2^*}{u_2}\right)\d x\\
&\ - \beta_h u_1^* u_2^* \int_\Omega \int_\Omega \Gamma(d_h\tau_b,x,y) g \left(\frac{(u_1\diamond u_2)(t-\tau_b,y)u_3^*}{u_1^*u_2^*u_3}\right)\d y\d x\\
&\ - \beta_h u_1^* u_2^* \int_\Omega \int_\Omega \Gamma(d_m\tau_a,x,y) g \left(\frac{u_1^*u_{3,-\tau_a}(y)}{u_1u_3^*}\right)\d y\d x,
\end{align*}
which lead to
\begin{align}\label{Equ1}
\frac{\d V(t)}{\d t} \leq &\ - \int_\Omega\beta_h u_1^* u_2^* g\left(\frac{u_2^*}{u_2}\right)\d x- \beta_h u_1^* u_2^* \int_\Omega \int_\Omega \Gamma(d_m\tau_a,x,y) g \left(\frac{u_1^*u_{3,-\tau_a}(y)}{u_1u_3^*}\right)\d y\d x\nonumber\\
&\ - \beta_h u_1^* u_2^* \int_\Omega \int_\Omega \Gamma(d_h\tau_b,x,y) g \left(\frac{(u_1\diamond u_2)(t-\tau_b,y)u_3^*}{u_1^*u_2^*u_3}\right)\d y\d x.
\end{align}
Hence, the map $t\mapsto V(u_t(\phi))$ is non-increasing.
{By some similar arguments in \cite{LiZhaoJDE2021,ZhangWangJMB2022},
we have $\lim_{t\rightarrow\infty} u(t,x,\phi) = u^*$.}
This completes the proof.

\end{proof}


\end{document}